\documentclass[12pt]{amsart}
\usepackage{listings} 
\usepackage{latexsym, amsmath, amssymb,amsthm,amsopn,amsfonts}
\usepackage{version}
\usepackage{mathrsfs}
\usepackage{epsfig,graphics,color,graphicx,graphpap}
\usepackage{amssymb}
\usepackage{todonotes}
\usepackage[citecolor=blue]{hyperref}
\usepackage{upgreek}

\usepackage[framemethod=tikz]{mdframed}

\newcommand{\redsout}{\bgroup\markoverwith{\textcolor{red}{\rule[0.5ex]{2pt}{.4pt}}}\ULon}
\newcommand{\curl}{\operatorname{curl}}

\newcommand{\p}{\partial}
\newcommand{\C}{\mathbb{C}}
\newcommand{\R}{\mathbb{R}}

\definecolor{skyblue}{rgb}{0.85,0.85,1}

\numberwithin{equation}{section}
\newtheorem{theorem}{Theorem}
\newtheorem{proposition}{Proposition}

\newtheorem{lemma}{Lemma}

\newtheorem{remark}{Remark}

\newcommand{\inclusion}{\hookrightarrow}

\newcommand{\Div}{\operatorname{Div}}

\renewcommand{\Re}{\operatorname{Re}}

\newcommand{\wt}{\widetilde}

\begin{document}
\title[Inverse problems for nonlinear Maxwell equations]{Inverse problems for nonlinear Maxwell's equations with second harmonic generation}

\author[Yernat M. Assylbekov]{Yernat M. Assylbekov }
\address{Department of Computational and Applied Mathematics, Rice University, Houston, TX 77005, USA}
\email{yernat.assylbekov@gmail.com}

\author[Ting Zhou]{Ting Zhou}
\address{Department of Mathematics, Northeastern University, Boston, MA 02115, USA}
\email{t.zhou@northeastern.edu}

\maketitle

\begin{abstract}
In the current paper we consider an inverse boundary value problem of electromagnetism with nonlinear Second Harmonic Generation (SHG) process. We show the unique determination of the electromagnetic material parameters and the SHG susceptibility parameter of the medium by making electromagnetic measurements on the boundary. We are interested in the case when a frequency is fixed.

\end{abstract}


\section{Introduction}\label{section::introduction}

The beginning of the field of {\em nonlinear optics} is often taken to be the discovery of second-harmonic generation by Franken et al. in 1961, shortly after the demonstration of the first working laser by Maiman in 1960. Nonlinear optical phenomena occur when the response of a material to an incident optical field depends nonlinearly on the strength of the optical field. For example, second-harmonic generation occurs as a result of the part of the atomic response that scales quadratically with the strength of the applied optical field. Consequently, the intensity of the light generated at the second-harmonic frequency tends to increase as the square of the intensity of the applied laser light. 

Ever since the invention of laser, fascinating new fields have emerged, among which {\em nonlinear optics} has the broadest scope. 
Applications of the nonlinear optical phenomena includes obtaining coherent radiation at a wavelength shorter than that of the incident laser, through the frequency doubling effect of SHG. Moreover, in the {\em second-harmonic imaging microscopy (SHIM)}, a second-harmonic microscope obtains contrasts from variations in {\em a specimen's ability to generate second-harmonic light} from the incident light while a conventional optical microscope obtains its contrast by detecting variations in optical density, path length, or refractive index of the specimen. 
The SHIM is also exploited in imaging flux residues (see the work in Chen Lab at the University of Michigan). Although nonlinear optical effects are in general very weak, the significant enhancement of SHG was shown using diffraction gratings or periodic structures, see \cite{RN,RNACMP}. \\

In this paper, we consider an inverse boundary value problem for Maxwell's equations with {second harmonic generation}. Let $\Omega$ be a bounded domain in $\R^3$ with smooth boundary. We consider the propagation of light through a nonlinear optical medium occupying $\Omega$ by beginning with Maxwell's equations without free charges
\[
\left\{\begin{array}{l}
\nabla\times E+\partial_t B=0\\
\nabla\times H-\partial_t D=0.
\end{array}
\right.
\]
where $E(t, x), H(t, x)$ are electric and magnetizing fields, $D$ is the electric displacement and $B$ is the magnetic field. For a non-magnetic material with nonlinear polarization, we have the constitutive relation 
\[B=\mu H\qquad D=\varepsilon E +P\]
where $\mu, \varepsilon$ are electric permittivity and magnetic permeability, respectively, and the polarization $P$ depends nonlinearly\footnote{There exists a linear term in the polarization, which we assume to be scalar and merge into the term $\varepsilon E$ so that our $P$ only has nonlinear terms.} on the electric field $E$. 
In the process of second-harmonic generation (SHG), 
the nonlinear polarization $P=(P_l)$ takes the form
\[P_l(t, x)=\sum_{jk}\chi_{jkl}^{(2)}E_jE_k\]
where $\chi^{(2)}$ is the second order susceptibility parameter. When a beam with time-harmonic electric field 
\[E(t, x)=E(x)e^{-i\omega t} + \textrm{c.c.}\]
is incident upon such a medium (e.g., a noncentrosymmetric crystal), new waves are generated at zero frequency and at frequency $2\omega$. 
Writing the solution $(E(t,x), H(t,x))$ to include these terms
\[\begin{split}E(t, x) &= 2\Re\left\{E^\omega(x)e^{-i\omega t}\right\}+2\Re\left\{E^{2\omega}(x)e^{-i2\omega t}\right\}\\
H(t, x) &= 2\Re\left\{H^\omega(x)e^{-i\omega t}\right\}+2\Re\left\{H^{2\omega}(x)e^{-i2\omega t}\right\},
\end{split}\]
and assuming that the susceptibility parameter is isotropic $\chi^{(2)}=(\chi^{(2)}_l)_{l=1}^3$,
then we obtain the time-harmonic system 
\begin{equation}\label{eqn:ME-SHG}
\left\{
\begin{split}
&\nabla\times E^{\omega}-i\omega\mu H^{\omega}=0,\\
&\nabla\times H^{\omega}+i\omega\varepsilon E^{\omega}+i\omega\chi^{(2)}\overline{E^{\omega}}\cdot{E^{2\omega}}=0,\\
&\nabla\times E^{2\omega}-i2\omega\mu H^{2\omega}=0,\\
&\nabla\times H^{2\omega}+i2\omega\varepsilon E^{2\omega}+i2\omega\chi^{(2)}E^{\omega}\cdot E^{\omega}=0.
\end{split}
\right.
\end{equation}
Here we can assume that $\mu$ and $\varepsilon$ depend on the frequency $\omega$.
For more discussion on the mathematical modeling of the second harmonic generation in nonlinear optics, we refer the readers to \cite{BD, BLZ, ABH}.

The inverse problem is to uniquely determine both the linear pair $(\mu, \varepsilon)$ and the second order susceptibility parameter $\chi^{(2)}$ using the boundary measurements of the electricmagnetic fields, taking the form of the admittance map defined in \eqref{eqn:admit_map}. Physically, this corresponds to measuring the tangential components of electric fields and magnetic fields for both frequencies $\omega$ and $2\omega$ on the boundary (surface) of the medium $\Omega$. \\

The type of inverse boundary value problem was first formulated by Calder\'on \cite{Calderon1980} for the linear conductivity equation $\nabla\cdot\gamma(x)\nabla u=0$ when he sought to determine the electrical conductivity $\gamma(x)$ of a medium by making boundary measurements of electric voltage and current. The unique determination was proved in \cite{SylvesterUhlmann1987} in dimension $n\geq3$ by solving the problem of determining an electric potential $q(x)$ in a Schr\"odinger operator $-\Delta+q$ from the boundary Dirichlet and Neumann data. Since then, the inverse problem has been extensively studied in various generalized cases, including those for other elliptic PDE's such as the magnetic Schr\"odinger equations, linear elasticity equations and so on. The inverse boundary value problem for linear time-harmonic Maxwell's equations was first formulated in \cite{SIC} and later fully solved in \cite{OPS1993, OlaSomersalo1996}, adopting a similar approach of implementing special complex phased exponential solutions, as in solving those for above mentioned scalar equations. 

In dealing with the inverse problems for nonlinear PDEs, such inverse boundary value problems have been considered for various semilinear and quasilinear elliptic equations and systems (see \cite{I1,I2,IS,IN,Su1,SuU,Su2,HSu}) based on a first order linearization approach. 
Recently the higher order linearization of the boundary map has been applied in determining the full nonlinearity of the medium for certain elliptic PDE's. See \cite{FO2019, KU201909, KU201905, LaiL2020, LLLS201903, LLLS201905}. The method was also successfully applied to solve inverse problems for nonlinear hyperbolic equations on the spacetime \cite{KLU18}, where in contrast the underlying problems for linear hyperbolic equations are still open. Shortly before these results, in our paper \cite{AssylbekovZhou2018}, we first adopt the method to solve the inverse problems for time-harmonic Maxwell's equations with Kerr-type nonlinearity, where the nonlinear term takes the form $\chi^{(3)}|E|^2E$ with $\chi^{(3)}(x)$ being the third order nonlinear susceptibility parameter and $E$ representing the electric field. In this type, the fields of different frequencies are not assumed coupled.\\

In this paper, we first present the well-posedness of the boundary value problem for the time-harmonic nonlinear second harmonic generation Maxwell's equations when the boundary data is small. Then we solve the inverse problem to determine the parameters, with the focus mainly on the second order susceptibility parameter $\chi^{(2)}(x)$. 
\subsection{Direct problem}
First we consider the boundary value problem for the nonlinear Maxwell equations~\eqref{eqn:ME-SHG}. We suppose that {$\varepsilon,\mu\in C^1(\Omega)$} are {complex-valued functions with positive real parts and {$\chi^{(2)}\in C^1(\Omega;\C^3)$}.\smallskip

The boundary conditions are expressed in terms of \emph{tangential trace},
$$
\mathbf{t}:C^\infty(\Omega; \C^3)\to C^\infty(\p \Omega; \C^3)\quad \mathbf{t}(w)=\nu\times w|_{\p\Omega},\quad w\in C^\infty(\Omega;\C^3),
$$
where $\nu$ is the unit outer normal vector to $\p\Omega$. Then $\mathbf t$ has its extension to a bounded operator $W^{1,p}(\Omega;\C^3)\to W^{1-1/p,p}(\p\Omega;\C^3)$ for $p>1$. Here and in what follows, $W^{1,p}(\Omega;\C^3)$ and $W^{1-1/p,p}(\p\Omega; \C^3)$ are standard Sobolev spaces on $\Omega$ and $\p \Omega$, respectively.\smallskip

To describe the boundary conditions, we introduce the spaces
\begin{align*}
W^{1,p}_{\Div}(\Omega)&=\{u\in W^{1,p}(\Omega;\C^3):\Div(\mathbf t(u))\in W^{1-1/p,p}(\p\Omega; \C^3)\},\\
TW^{1-1/p,p}_{\Div}(\p \Omega)&=\{f\in W^{1-1/p,p}(\p\Omega;\C^3):\Div(f)\in W^{1-1/p,p}(\p\Omega;\C^3)\},
\end{align*}
where $\Div$ is the surface divergence on $\p \Omega$. These spaces are Banach spaces with norms
\begin{align*}
\|u\|_{W^{1,p}_{\Div}(\Omega)}&=\|u\|_{W^{1,p}(\Omega;\C^3)}+\|\Div(\mathbf t(u))\|_{W^{1-1/p,p}(\p \Omega;\C^3)},\\
\|u\|_{TW^{1-1/p,p}_{\Div}(\p M)}&=\|f\|_{W^{1-1/p,p}(\p M)}+\|\Div(f)\|_{W^{1-1/p,p}(\p M)}.
\end{align*}
It is not difficult to see that $\mathbf t(W^{1,p}_{\Div}(M))=TW^{1-1/p,p}_{\Div}(\p M)$.
Our first main result is the following theorem on well-posedness of the nonlinear Maxwell equations~\eqref{eqn:ME-SHG} with prescribed small $\mathbf t(E^{\omega})$ and $\mathbf t(E^{2\omega})$ on $\p \Omega$.

\begin{theorem}\label{thm::main result 1}
Let $\Omega$ be a bounded domain in $\R^3$ with smooth boundary and let $3<p\le 6$. Suppose that {$\varepsilon,\mu\in C^1(\Omega;\C)$} are complex-valued functions with positive real parts and $\chi^{(2)}\in C^1(\Omega;\R^3)$. For every $\omega\in \C$, outside a discrete set $\Sigma\subset\C$ of resonant frequencies, there is $\epsilon>0$ such that for a pair $(f^\omega,f^{2\omega})\in \left(TW^{1-1/p,p}_{\Div}(\p M)\right)^2$ with $\displaystyle\sum_{k=1,2}\|f^{k\omega}\|_{TW^{1-1/p,p}_{\Div}(\p \Omega)}<\epsilon$, the Maxwell's equations \eqref{eqn:ME-SHG} has a unique solution $(E^{\omega},H^{\omega}, E^{2\omega}, H^{2\omega})\in \left(W^{1,p}_{\Div}(\Omega)\right)^4$ satisfying $\mathbf{t}(E^{k\omega})=f^{k\omega}$ for $k=1,2$ and
$$
\sum_{k=1,2}\|E^{k\omega}\|_{W^{1,p}_{\Div}(\Omega)}+\|H^{k\omega}\|_{W^{1,p}_{\Div}(\Omega)}\le C\sum_{k=1,2}\|f^{k\omega}\|_{TW^{1-1/p,p}_{\Div}(\p \Omega)},
$$
for some constant $C>0$ independent of $(f^{\omega}, f^{2\omega})$.
\end{theorem}

\begin{remark}
{ We can also add nonlinear magnetic polarization terms to the equations. The analysis will be the same.
}	
\end{remark}


\subsection{Inverse problem}

For $\omega>0$ with $\omega\notin \Sigma$, we define the \emph{admittance map} $\Lambda_{\varepsilon,\mu,\chi^{(2)}}^{\omega,2\omega}$ as
\begin{equation}\label{eqn:admit_map}
\Lambda_{\varepsilon,\mu,\chi^{(2)}}^{\omega, 2\omega}(f^\omega, f^{2\omega})=\left(\mathbf t(H^{\omega}), \mathbf t(H^{2\omega})\right),
\end{equation}
for $(f^\omega, f^{2\omega})\in \left(TW^{1-1/p,p}_{\Div}(\p \Omega)\right)^2$ with $\displaystyle\sum_{k=1,2}\|f^{k\omega}\|_{TW^{1-1/p,p}_{\Div}(\p \Omega)}<\epsilon$, 
where $(E^{\omega},H^{\omega}, E^{2\omega}, H^{2\omega})\in \left(W^{1,p}_{\Div}(\Omega)\right)^4$ is the unique solution of the system \eqref{eqn:ME-SHG} with $\mathbf t(E^{k\omega})=f^{k\omega}$, ($k=1,2$), guaranteed by Theorem~\ref{thm::main result 1}. Moreover, the estimate provided in Theorem~\ref{thm::main result 1} implies that the admittance map satisfy
$$
\|\Lambda_{\varepsilon,\mu,\chi^{(2)}}^{\omega,2\omega}(f^\omega, f^{2\omega})\|_{\left(TW^{1-1/p,p}_{\Div}(\p \Omega)\right)^2}\le C\sum_{k=1,2}\|f^{k\omega}\|_{TW^{1-1/p,p}_{\Div}(\p \Omega)}<C\epsilon.
$$
The inverse problem is to determine $\varepsilon,\mu$ and $\chi^{(2)}$ from the knowledge of the admittance map $\Lambda_{\varepsilon,\mu,\chi^{(2)}}^{\omega, 2\omega}$.\smallskip\\

Our second main result is for the inverse problem as follows.

\begin{theorem}\label{thm::main result 2}
Let $\Omega$ be a bounded domain in $\R^3$ with smooth boundary and let $3< p< 6$. Suppose that {$\varepsilon_j,\mu_j\in C^3(\Omega;\C)$} with positive real parts and {$\chi^{(2)}_j\in C^1(\Omega;\C^3)$}, $j=1,2$. Fix $\omega>0$ outside a discrete set of resonant frequencies $\Sigma\subset\C$ and fix sufficiently small $\epsilon>0$. If
$$
\Lambda_{\varepsilon_1,\mu_1,\chi^{(2)}_1}^{\omega,2\omega}(f^\omega, f^{2\omega})=\Lambda_{\varepsilon_2,\mu_2,\chi^{(2)}_2}^{\omega,2\omega}(f^\omega, f^{2\omega})
$$
for all $(f^\omega, f^{2\omega})\in \left(TW^{1-1/p,p}_{\Div}(\p \Omega)\right)^2$ with $\displaystyle\sum_{k=1,2}\|f^{k\omega}\|_{TW^{1-1/p,p}_{\Div}(\p \Omega)}<\epsilon$, then $\varepsilon_1=\varepsilon_2$, $\mu_1=\mu_2$, $\chi^{(2)}_1=\chi^{(2)}_2$ in $\Omega$.
\end{theorem}

For the SHG type of nonlinear Maxwell's equations, after first order linearization of the boundary admittance map with respect to the boundary input, we can quickly recover $\mu$ and $\varepsilon$ by solving the corresponding inverse problem for the linear equations (see \cite{OPS1993, OlaSomersalo1996}). The difficulty lies in reconstructing the susceptibility parameter $\chi^{(2)}$. By calculating the second order term of the asymptotic expansion of the admittance map, one derives a key integral identity for the parameter $\chi^{(2)}$ multiplied by the product of three linear solutions. The general idea here is to plug in the complex geometrical optics (CGO) solutions that are constructed in solving the inverse problems for linear equations, in order to obtain the Fourier transform of the parameter. When the nonlinearity is of Kerr-type as discussed in \cite{AssylbekovZhou2018}, the integral identity involves products of the parameter $\chi^{(3)}$ with four linear solutions. By plugging in the CGO solutions constructed as in \cite{KenigSaloUhlmann2011a} or \cite{ChungOlaSaloTzou2016} using the reduction to a second order Helmholtz Schr\"odinger system technique, the sum of four complex phases and the product of four amplitudes can be managed to give the desirable Fourier transform. However, with products of three solutions as in the SHG case, the special formats of these solutions do not allow us to choose proper complex phases and amplitudes for the CGO's to get the Fourier transform from the product. Instead, we turn to \cite{OPS1993} for their CGO solutions, with carefully chosen complex phases and amplitudes, and we could successfully reduce the integral into the Fourier transform of components of $\chi^{(2)}$. \\

The paper is organized as following. In Section \ref{sec:directproblem}, we prove the well-posedness of the direct problem (Theorem \ref{thm::main result 1}). In Section \ref{sec:asymp}, we compute the asymptotic expansion of the admittance map $\Lambda^{\omega, 2\omega}_{\mu,\varepsilon, \chi^{(2)}}$. To solve the inverse problem, the reconstruction of $\mu$ and $\varepsilon$ is a direct corollary of the result for linear equations. Then we prove the unique determination of $\chi^{(2)}$ in Section \ref{sec:chi_2}, which completes the proof of Theorem \ref{thm::main result 2}.  \\

\textbf{Acknowledgements.} The research of TZ is supported by the NSF grant DMS-1501049. YA was supported by Total S.A. and the members of the Geo-Mathematical Imaging Group at Rice University. The authors would like to thank Prof. Gunther Uhlmann for suggesting the problem and to thank Prof. Gang Bao for helpful discussions on the second harmonic generation. 


%
%
%


\section{{ Well-posedness of the direct problem}}\label{sec:directproblem}

To prove existence and uniqueness result for the nonlinear  equations, we first need the results for both the homogeneous and inhomogeneous linear equations, given in \cite[Section~3.1]{AssylbekovZhou2018}.

\begin{lemma}\label{thm:linear-homog}
Let $\Omega$ be a bounded domain with a smooth boundary and $2\le p\le 6$ and let $\varepsilon,\mu\in C^1(\Omega;\C)$ be complex functions with positive real parts. There is a discrete subset $\Sigma$ of $\C$ such that for all $\omega\notin \Sigma$ and for a given $f\in TW^{1-1/p,p}_{\Div}(\p \Omega)$ the Maxwell's equation 
\[
\nabla\times E-i\omega\mu H=0,\quad
\nabla\times H+i\omega \varepsilon E=0
\]
has a unique solution $(E,H)\in W^{1,p}_{\Div}(\Omega)\times W^{1,p}_{\Div}(\Omega)$ satisfying $\mathbf{t}(E)=f$ and 
$$
\|E\|_{W^{1,p}_{\Div}(\Omega)}+\|H\|_{W^{1,p}_{\Div}(\Omega)}\le C\|f\|_{TW^{1-1/p,p}_{\Div}(\p \Omega)},
$$
for some constant $C>0$ independent of $f$.
\end{lemma}

For the boundary value problem of inhomogeneous linear equations, we define
$$
W^{1,p}_D(\Omega):=\{u\in W^{1,p}(\Omega;\C^3):\mathbf t(u)=0\}.
$$
and
\[
\begin{split}
W^p(\nabla\times,\Omega)&:=\{u\in L^p(\Omega;\C^3):~\nabla\times u\in L^p(\Omega;\C^3)\},\\
W^p(\nabla\cdot,\Omega)&:=\{u\in L^p(\Omega;\C^3):~\nabla\cdot u\in L^p(\Omega;\C)\}
\end{split}
\]
endowed with the norms {
\[\begin{split}
\|u\|_{W^p(\nabla\times,\Omega)}&:=\|u\|_{L^p(\Omega;\C^3)}+\|\nabla\times u\|_{L^p(\Omega;\C^3)},\\
\|u\|_{W^p(\nabla\cdot,\Omega)}&:=\|u\|_{L^p(\Omega;\C^3)}+\|\nabla\cdot u\|_{L^p(\Omega;\C)}.\end{split}\]
}
Then we have

\begin{lemma}\label{thm:linear-inhomog}
Let $2\le p\le 6$ and let {$\varepsilon,\mu\in C^1(\Omega;\C)$} be complex-valued functions with positive real parts. Suppose that $J_e,J_m\in W^p(\nabla\cdot,\Omega)$ and $\nu \cdot J_e|_{\p \Omega}, \nu\cdot J_m|_{\p \Omega}\in W^{1-1/p,p}(\p \Omega)$. There is a discrete subset $\Sigma$ of $\C$ such that for all $\omega\notin \Sigma$ the Maxwell's system
\[
\nabla\times E-i\omega\mu H=J_m,\quad
\nabla\times H+i\omega \varepsilon E=J_e
\]
has a unique solution $(E,H)\in W^{1,p}_D(\Omega)\times W^{1,p}_{\Div}(\Omega)$ satisfying 
\begin{align*}
\|E\|_{W^{1,p}_{\Div}(\Omega)}+\|H\|_{W^{1,p}_{\Div}(\Omega)}&\le C(\|\nu\cdot J_e|_{\p \Omega}\|_{W^{1-1/p,p}(\p \Omega)}+\|\nu\cdot J_m|_{\p \Omega}\|_{W^{1-1/p,p}(\p \Omega)})\\
&\qquad+C(\|J_e\|_{W^{p}(\nabla\cdot,\Omega)}+\|J_m\|_{W^{p}(\nabla\cdot, \Omega)})
\end{align*}
for some constant $C>0$ independent of $J_e$ and $J_m$.
\end{lemma}


In addition, using the Sobolev embedding $W^{1,p}(\Omega)\inclusion C(\Omega)$ when $p>n$, we obtain
\[\begin{split}
\|fg\|_{W^{1,p}(\Omega)}&=\|fg\|_{L^p(\Omega)}+\|f\nabla g+g\nabla f\|_{L^p(\Omega)}\\
&\leq C\left\{\|f\|_{L^\infty(\Omega)}\left(\|g\|_{L^p(\Omega)}+\|\nabla g\|_{L^p(\Omega)}\right)+\|g\|_{L^\infty(\Omega)}\|\nabla f\|_{L^p(\Omega)}\right\}\\
&\leq C\|f\|_{W^{1,p}(\Omega)}\|g\|_{W^{1,p}(\Omega)}.
\end{split}\]
Therefore, when {$p>3$}, $W^{1,p}(\Omega;\C^3)$ is an algebra and the nonlinear terms in Maxwell's equations \eqref{eqn:ME-SHG} are in appropriate function spaces. \\

\subsection*{Proof of Theorem~\ref{thm::main result 1}}

Suppose $f^{k\omega}\in TW^{1-1/p,p}_{\Div}(\p \Omega)$ with $k=1,2$ such that 
$$\displaystyle\sum_{k=1,2}\|f^{k\omega}\|_{TW^{1-1/p,p}_{\Div}(\p \Omega)}<\epsilon,$$ with $\epsilon>0$ to be determined. By Lemma~\ref{thm:linear-homog}, there is a unique $(E_0^{k\omega},H_0^{k\omega})\in W^{1,p}_{\Div}(\Omega)\times W^{1,p}_{\Div}(\Omega)$ solving
$$
\nabla\times E^{k\omega}_0-ik\omega\mu H^{k\omega}_0=0,\quad \nabla\times H^{k\omega}_0+ik\omega \varepsilon E^{k\omega}_0=0,\quad \mathbf t(E^{k\omega}_0)=f^{k\omega},\quad k=1,2.
$$
and satisfying
$$
\sum_{k=1,2}\|E^{k\omega}_0\|_{W^{1,p}_{\Div}(\Omega)}+\|H^{k\omega}_0\|_{W^{1,p}_{\Div}(\Omega)}\le C\sum_{k=1,2}\|f^{k\omega}\|_{TW^{1-1/p,p}_{\Div}(\p \Omega)}.
$$
Then $(E^\omega,H^\omega, E^{2\omega}, H^{2\omega})$ is a solution to \eqref{eqn:ME-SHG} if and only if $({E^\omega}',{H^\omega}',{E^{2\omega}}', {H^{2\omega}}')$ defined by $(E^{k\omega},H^{k\omega})=(E^{k\omega}_0,H^{k\omega}_0)+({E^{k\omega}}',{H^{k\omega}}')$ for $k=1,2$ satisfies
\begin{equation}\label{eqn:ME-SHG-Dirichlet}
\begin{cases}
\nabla\times {E^\omega}'-i\omega\mu {H^\omega}'=0,\\
\nabla\times {H^\omega}'+i\omega\varepsilon {E^\omega}'+i\omega\chi^{(2)}\left(\overline{E_0^\omega+{E^\omega}'}\right)\cdot\left(E_0^{2\omega}+{E^{2\omega}}'\right)=0,\\
\nabla\times {E^{2\omega}}'-i2\omega\mu {H^{2\omega}}'=0,\\
\nabla\times {H^{2\omega}}'+i2\omega \varepsilon E'+i2\omega \chi^{(2)}\left(E_0^\omega+{E^\omega}'\right)\cdot\left(E_0^\omega+{E^\omega}'\right)=0,\\
\mathbf t({E^{k\omega}}')=0,\qquad k=1,2.
\end{cases}
\end{equation}
By Lemma~\ref{thm:linear-inhomog}, there is a bounded linear operator
$$
\mathcal G_{k\omega}^{\varepsilon,\mu}:W^{1,p}(\Omega;\C^3)\times W^{1,p}(\Omega; \C^3)\to W^{1,p}_D(\Omega)\times W^{1,p}_{\Div}(\Omega)
$$
mapping $(J_e,J_m)\in W^{1,p}(\Omega;\C^3)\times W^{1,p}(\Omega; \C^3)$ to the unique solution $(\widetilde E,\widetilde H)$ of the problem
$$
\nabla\times\widetilde E-ik\omega\mu \widetilde H=J_m,\quad \nabla\times\widetilde H+ik\omega \varepsilon \widetilde E=J_e,\quad \mathbf t(\widetilde E)=0.
$$
Define $X_\delta$ to be the set of $(e^\omega,h^\omega, e^{2\omega}, h^{2\omega})\in W^{1,p}_D(\Omega)\times W^{1,p}_{\Div}(\Omega)\times W^{1,p}_D(\Omega)\times W^{1,p}_{\Div}(\Omega)$ such that
$$
\sum_{k=1,2}\|e^{k\omega}\|_{W^{1,p}(\Omega;\C^3)}+\|h^{k\omega}\|_{W^{1,p}_{\Div}(\Omega)}\le\delta,
$$
where $\delta>0$ will be determined later. Define an operator $A$ on $X_\delta$ as
$$
A(e^\omega, h^\omega, e^{2\omega}, h^{2\omega}):=\left(\mathcal G_\omega^{\varepsilon,\mu}\big(0, i\omega \chi^{(2)} (\overline{E_0^\omega+e^\omega})\cdot(E_0^{2\omega}+e^{2\omega})\big),  \mathcal G_{2\omega}^{\varepsilon,\mu}\big(0, i2\omega \chi^{(2)} (E_0^{\omega}+e^{\omega})\cdot(E_0^{\omega}+e^{\omega})\big)                                                                                                                                                                                                                                                                                                                                                                                                                                                                                                                                                                                                                                                                                                                                                                                                                                                                                                                                                                                                                                                                                                                                                                                                                                                                                                                                                                                                                                                                                                                                                                                                                                                                                                                                                                                                                                                                                                                                                                                                                                                                                                                                                                                                                                                                                                                                                                                                                                                                                                                                                                                                                                                                                                                                                                                                                                                                                                                                                                                                                                                                                                                                                                                                                                                                                                                                                                                                                                                                                                                                                                                                                                                                                                                                                                                                                                                                                                                                                                                                                                                                                                                                                                                                                                                                                                                                                                                                                                                                                                                                                                                                                                                                                                                                                                                                                                                                                                                                                                                                                                                                                                                                                                                                                                                                                                                                                                                                                                                                        \right).
$$

We wish to show that for sufficiently small $\epsilon>0$ and $\delta>0$, depending on the frequency $\omega$, the operator $A$ is a contraction on $X_\delta$.\smallskip

\begin{remark}{\rm
{Note that the operator does not depend on $h^\omega$ and $h^{2\omega}$ in this case. However, it can be generalized to cover the case with second order magnetic nonlinearity. }\\}
\end{remark}

{First, $A$ maps $X_\delta$ into itself.} 
Indeed, when {$p>3$}, 
we have
\begin{align*}
\|A(e^\omega,h^\omega, &e^{2\omega}, h^{2\omega})\|_{W^{1,p}_D(\Omega)\times W^{1,p}_{\Div}(\Omega)\times W^{1,p}_D(\Omega)\times W^{1,p}_{\Div}(\Omega)}\\
&\le C\omega\big(\|\chi^{(2)}(\overline{E_0^\omega+e^\omega})\cdot(E_0^{2\omega}+e^{2\omega})\|_{W^{1,p}(\Omega; \C^3)}+\|\chi^{(2)} (E_0^{\omega}+e^{\omega})\cdot(E_0^{\omega}+e^{\omega})\|_{W^{1,p}(\Omega; \C^3)}\big)\\
&\le C\omega\|E_0^\omega+e^\omega\|_{W^{1,p}(\Omega; \C^3)}\big(\|E_0^{2\omega}+e^{2\omega}\|_{W^{1,p}(\Omega; \C^3)}+\|E_0^\omega+e^\omega\|_{W^{1,p}(\Omega; \C^3)}\big)\\
&\le C\omega\big(\|E_0^\omega\|^2_{W^{1,p}(\Omega; \C^3)}+\|E_0^{2\omega}\|^2_{W^{1,p}(\Omega; \C^3)}+\|e^\omega\|^2_{W^{1,p}(\Omega; \C^3)}+\|e^{2\omega}\|^2_{W^{1,p}(\Omega; \C^3)}\big).
\end{align*}
Therefore,
\begin{equation}\label{eqn:control for A}
\begin{aligned}
\|A(e^\omega,h^\omega, &e^{2\omega}, h^{2\omega})\|_{W^{1,p}_D(\Omega)\times W^{1,p}_{\Div}(\Omega)\times W^{1,p}_D(\Omega)\times W^{1,p}_{\Div}(\Omega)}\\
&\le C\omega\left(\epsilon\sum_{k=1,2}\|f^{k\omega}\|_{TW^{1-1/p,p}_{\Div}(\p\Omega)}+\delta\|(e^\omega, h^\omega, e^{2\omega}, h^{2\omega})\|_{W^{1,p}_D(\Omega)\times W^{1,p}_{\Div}(\Omega)\times W^{1,p}_D(\Omega)\times W^{1,p}_{\Div}(\Omega)}\right)\\
&\le C\omega(\epsilon^{2}+\delta^{2}).
\end{aligned}
\end{equation}
Taking $\epsilon>0$ and $\delta>0$ sufficiently small ensures that $A$ maps $X_\delta$ into itself. \\ \smallskip

{Next, we show that $A$ is a contraction on $X_\delta$.} For $(e^\omega_1, h^\omega_1, e^{2\omega}_1, h^{2\omega}_1), (e^\omega_2, h^\omega_2, e^{2\omega}_2, h^{2\omega}_2)\in X_\delta$, we have
\[\begin{split}
\|i\omega\chi^{(2)}(\overline{E_0^\omega+e_2^\omega})&\cdot(E_0^{2\omega}+e^{2\omega}_2)-i\omega\chi^{(2)}(\overline{E_0^\omega+e_1^\omega})\cdot(E_0^{2\omega}+e^{2\omega}_1)\|_{W^{1,p}(\Omega;\C^3)}\\
&\leq C\omega\|(\overline{E_0^{\omega}+e^\omega_2})\cdot(e_2^{2\omega}-e_1^{2\omega})+(\overline{e^\omega_2-e^\omega_1})\cdot(E^{2\omega}_0+e^{2\omega}_1)\|_{W^{1,p}(\Omega; \C^3)},
\end{split}\]  
and
\[\begin{split}
\|i2\omega\chi^{(2)}({E_0^\omega+e_2^\omega})&\cdot(E_0^{\omega}+e^{\omega}_2)-i2\omega\chi^{(2)}({E_0^\omega+e_1^\omega})\cdot(E_0^{\omega}+e^{\omega}_1)\|_{W^{1,p}(\Omega;\C^3)}\\
&\leq C\omega\| (2 E_0^\omega+e^\omega_2+e^\omega_1)\cdot(e_2^\omega-e_1^\omega) \|_{W^{1,p}(\Omega; \C^3)}.
\end{split}\]  
Together, these imply
\[\begin{split}
\|A(e_2^\omega, h_2^\omega, e_2^{2\omega}, h_2^{2\omega})&-A(e_1^\omega, h_1^\omega, e_1^{2\omega}, h_1^{2\omega})\|_{W^{1,p}_D(\Omega)\times W^{1,p}_{\Div}(\Omega)\times W^{1,p}_D(\Omega)\times W^{1,p}_{\Div}(\Omega)}\\
\leq C\omega &\Big\{(\|E_0^\omega\|_{W^{1,p}(\Omega; \C^3)}+\|e^\omega_2\|_{W^{1,p}(\Omega; \C^3)})\|e_2^{2\omega}-e_1^{2\omega}\|_{W^{1,p}(\Omega; \C^3)}\\
&+\left(\|E_0^{2\omega}+e_1^{2\omega}\|_{W^{1,p}(\Omega; \C^3)}+\|2E_0^\omega+e_2^\omega+e_1^\omega\|_{W^{1,p}(\Omega; \C^3)}\right)\|e_2^\omega-e_1^\omega\|_{W^{1,p}(\Omega; \C^3)}\Big\}\\
\leq C\omega &\left(\sum_{k=1,2}\|E_0^{k\omega}\|_{W^{1,p}(\Omega;\C^3)}+\sum_{k,j=1,2}\|e_j^{k\omega}\|_{W^{1,p}(\Omega;\C^3)}\right)\sum_{k=1,2}\|e^{k\omega}_2-e^{k\omega}_1\|_{W^{1,p}(\Omega;\C^3)}\\
\leq C\omega &(\epsilon+\delta)\|(e^\omega_2, h^\omega_2, e^{2\omega}_2, h^{2\omega}_2)-(e^\omega_1, h^\omega_1, e^{2\omega}_1, h^{2\omega}_1)\|_{W^{1,p}_D(\Omega)\times W^{1,p}_{\Div}(\Omega)\times W^{1,p}_D(\Omega)\times W^{1,p}_{\Div}(\Omega)}.
\end{split}\]
This verifies that $A$ is a contraction when $C\omega(\epsilon+\delta)<1$. \\

Now, using the contraction mapping theorem, there exists a unique fixed point $$({E^\omega}',{H^\omega}', {E^{2\omega}}', {H^{2\omega}}')\in X_\delta$$ of $A$. It solves \eqref{eqn:ME-SHG-Dirichlet} when {$3< p\leq 6$} and $\delta>0, \epsilon>0$ are small enough. Using \eqref{eqn:control for A}, one can see that
\[\begin{split}
\|({E^\omega}',&{H^\omega}', {E^{2\omega}}', {H^{2\omega}}')\|_{W^{1,p}_D(\Omega)\times W^{1,p}_{\Div}(\Omega)\times W^{1,p}_D(\Omega)\times W^{1,p}_{\Div}(\Omega)}\\
&= \|A({E^\omega}',{H^\omega}', {E^{2\omega}}', {H^{2\omega}}')\|_{W^{1,p}_D(\Omega)\times W^{1,p}_{\Div}(\Omega)\times W^{1,p}_D(\Omega)\times W^{1,p}_{\Div}(\Omega)}\\
&\leq C\omega\left(\epsilon\sum_{k=1,2}\|f^{k\omega}\|_{TW^{1-1/p,p}_{\Div}(\p\Omega)}+\delta\|({E^\omega}',{H^\omega}', {E^{2\omega}}', {H^{2\omega}}')\|_{W^{1,p}_D(\Omega)\times W^{1,p}_{\Div}(\Omega)\times W^{1,p}_D(\Omega)\times W^{1,p}_{\Div}(\Omega)}\right).
\end{split}\]
Choosing $C\omega\delta<\frac12$, we obtain the estimate
\[\|({E^\omega}',{H^\omega}', {E^{2\omega}}', {H^{2\omega}}')\|_{W^{1,p}_D(\Omega)\times W^{1,p}_{\Div}(\Omega)\times W^{1,p}_D(\Omega)\times W^{1,p}_{\Div}(\Omega)}\leq C\sum_{k=1,2}\|f^{k\omega}\|_{TW^{1-1/p,p}_{\Div}(\p\Omega)}.\]
Finally, $(E^\omega,H^\omega, E^{2\omega}, H^{2\omega})=(E^\omega_0,H^\omega_0, E^{2\omega}_0, H^{2\omega}_0)+({E^\omega}',{H^\omega}', {E^{2\omega}}', {H^{2\omega}}')$ solves \eqref{eqn:ME-SHG} with $\mathbf t(E^{k\omega})=f^{k\omega}$ ($k=1,2$) and satisfies the estimate
$$
\sum_{k=1,2}\|E^{k\omega}\|_{W^{1,p}_{\Div}(\Omega)}+\|H^{k\omega}\|_{W^{1,p}_{\Div}(\Omega)}\le C\sum_{k=1,2}\|f^{k\omega}\|_{TW^{1-1/p,p}_{\Div}(\p \Omega)}.
$$
The proof of Theorem~\ref{thm::main result 1} is thus complete.



\section{Asymptotics of the admittance map}\label{sec:asymp}
Let $\Omega$ be a bounded domain in $\R^3$ with smooth boundary and let $3<p\le 6$. Suppose that $\varepsilon,\mu\in C^1(\Omega;\C)$ are complex functions with positive real parts and {$\chi^{(2)}\in C^1(\Omega;\C^3)$}. Fix $\omega>0$ outside a discrete set of resonant frequencies. Suppose that $(f^\omega, f^{2\omega})\in TW^{1-1/p,p}_{\Div}(\p \Omega)\times TW^{1-1/p,p}_{\Div}(\p \Omega)$ and $s\in\R$ is small enough. By Theorem~\ref{thm::main result 1}, there is a unique solution $(E^{\omega}_s,H^{\omega}_s, E^{2\omega}_s, H^{2\omega}_s)\in \left(W^{1,p}_{\Div}(\Omega)\right)^4$ of \eqref{eqn:ME-SHG} such that $\mathbf t(E^{k\omega}_s)=sf^{k\omega}$ with $k=1,2$ and
\begin{equation}\label{eqn::estimate for E^s and H^s}
\sum_{k=1,2}\|E^{k\omega}_s\|_{W^{1,p}_{\Div}(\Omega)}+\|H^{k\omega}_s\|_{W^{1,p}_{\Div}(\Omega)}\le C|s|\sum_{k=1,2}\|f^{k\omega}\|_{TW^{1-1/p,p}_{\Div}(\p \Omega)}.
\end{equation}
By Lemma~\ref{thm:linear-homog}, there is a unique $(E_1^\omega,H_1^\omega,E_1^{2\omega}, H_1^{2\omega})\in \left(W^{1,p}_{\Div}(\Omega)\right)^4$ solving 
$$
\nabla\times E^{k\omega}_1-ik\omega\mu H^{k\omega}_1=0,\quad \nabla\times H^{k\omega}_1+ik\omega \varepsilon E^{k\omega}_1=0,\quad \mathbf t(E^{k\omega}_1)=f^{k\omega},\quad k=1,2.
$$
such that
\begin{equation}\label{eqn::estimate for E_1 and H_1}
\sum_{k=1,2}\|E_1^{k\omega}\|_{W^{1,p}_{\Div}(\Omega)}+\|H_1^{k\omega}\|_{W^{1,p}_{\Div}(\Omega)}\le C\sum_{k=1,2}\|f^{k\omega}\|_{TW^{1-1/p,p}_{\Div}(\p \Omega)}.
\end{equation}
Also, by Lemma~\ref{thm:linear-inhomog} there is a unique solution $(E^\omega_2,H^\omega_2,E^{2\omega}_2, H^{2\omega}_2)\in W^{1,p}_{D}(\Omega)\times W^{1,p}_{\Div}(\Omega)\times W^{1,p}_{D}(\Omega)\times W^{1,p}_{\Div}(\Omega)$ for
\[\left\{\begin{split}
&\nabla\times E^{k\omega}_2-ik\omega\mu H^{k\omega}_2=0,\quad k=1,2,\\
&\nabla\times H^\omega_2+i\omega \varepsilon E^\omega_2+i\omega \chi^{(2)}\overline{E^\omega_1}\cdot E^{2\omega}_1=0,\\
&\nabla\times H^{2\omega}_2+i2\omega\varepsilon E^{2\omega}_2+i\omega\chi^{(2)}E^\omega_1\cdot E^\omega_1=0.
\end{split}\right.\]
satisfying
\[\begin{split}
\sum_{k=1,2}\|E^{k\omega}_2\|_{W^{1,p}_{\Div}(\Omega)}+\|H^{k\omega}_2\|_{W^{1,p}_{\Div}(\Omega)}&\le C\left(\|\,|\chi^{(2)}\overline{E^\omega_1}\cdot E^{2\omega}_1\|_{W^{1,p}(\Omega;\C^3)}+\|\chi^{(2)}E^\omega_1\cdot E^\omega_1\|_{W^{1,p}(\Omega;\C^3)}\right)\\
&\le C\left(\|E^\omega_1\|_{W^{1,p}(\Omega;\C^3)}\|E^{2\omega}_1\|_{W^{1,p}(\Omega;\C^3)}+\|E^\omega_1\|_{W^{1,p}(\Omega;\C^3)}^2\right),
\end{split}\]
hence
\begin{equation}\label{eqn:est-E_2-H_2}
\sum_{k=1,2}\|E^{k\omega}_2\|_{W^{1,p}_{\Div}(\Omega)}+\|H^{k\omega}_2\|_{W^{1,p}_{\Div}(\Omega)}\le C\sum_{k=1,2}\|f^{k\omega}\|^2_{TW^{1-1/p,p}_{\Div}(\p \Omega)}.\end{equation}

Now we define $(F^\omega_s, G^\omega_s, F^{2\omega}_s, G^{2\omega}_s)\in W^{1,p}_D(\Omega)\times W^{1,p}_{\Div}(\Omega)\times W^{1,p}_D(\Omega)\times W^{1,p}_{\Div}(\Omega)$ by
\begin{equation}\label{eqn:def-F-G}
(E^{k\omega}_s,H^{k\omega}_s)=s(E_1^{k\omega}+sF^{k\omega}_s,H_1^{k\omega}+sG^{k\omega}_s),\qquad k=1,2.
\end{equation}

First,  by \eqref{eqn::estimate for E^s and H^s} and \eqref{eqn::estimate for E_1 and H_1}, it satisfies
\begin{align*}
\sum_{k=1,2}|s|^{2}\|F^{k\omega}_s&\|_{W^{1,p}(\Omega;\C^3)}+|s|^{2}\|G^{k\omega}_s\|_{W^{1,p}_{\Div}(\Omega)}\\
&\le \sum_{k=1,2}\|E^{k\omega}_s\|_{W^{1,p}_{\Div}(\Omega)}+\|H^{k\omega}_s\|_{W^{1,p}_{\Div}(\Omega)}+|s|\,\|E^{k\omega}_1\|_{W^{1,p}_{\Div}(\Omega)}+|s|\,\|H^{k\omega}_1\|_{W^{1,p}_{\Div}(\Omega)}\\
&\le C|s|\sum_{k=1,2}\|f^{k\omega}\|_{TW^{1-1/p,p}_{\Div}(\p \Omega)}.
\end{align*}
Therefore,
\begin{equation}\label{eqn::est-F-G}
\sum_{k=1,2}|s|\|F^{k\omega}_s\|_{W^{1,p}(\Omega;\C^3)}+|s|\|G^{k\omega}_s\|_{W^{1,p}_{\Div}(\Omega)}\le C\sum_{k=1,2}\|f^{k\omega}\|_{TW^{1-1/p,p}_{\Div}(\p \Omega)}.
\end{equation}

Next, plugging \eqref{eqn:def-F-G} in the nonlinear Maxwell's equations, we obtain
\[\left\{\begin{split}
&\nabla\times F^{k\omega}_s-ik\omega\mu G^{k\omega}_s=0,\qquad k=1,2,\\
&\nabla\times G^\omega_s+i\omega\varepsilon F^\omega_s+i\omega\chi^{(2)}\left(\overline{E^\omega_1}\cdot E^{2\omega}_1+s\overline{E^\omega_1}\cdot F^{2\omega}_s+s\overline{F^\omega_s}\cdot E^{2\omega}_1+s^2\overline{F^\omega_s}\cdot F^{2\omega}_s\right)=0,\\
&\nabla\times G^{2\omega}_s+i2\omega\varepsilon F^{2\omega}_s+i2\omega\chi^{(2)}\left(E^\omega_1\cdot E^\omega_1+2s E^\omega_1\cdot F^\omega_s+s^2 F^\omega_s\cdot F^\omega_s\right)=0,\\
&\mathbf t(F^{k\omega}_s) = 0,\qquad k=1,2.
\end{split}\right.\]

Set 
\begin{equation}\label{eqn:def-P-Q}
P_s^{k\omega}:=F^{k\omega}_s-E_2^{k\omega}, \quad Q_s^{k\omega}:=G^{k\omega}_s-H_2^{k\omega},\qquad k=1,2.\end{equation}
Then $(P_s^\omega, Q_s^\omega, P_s^{2\omega}, Q_s^{2\omega})$ satisfies 
\[\left\{\begin{split}
&\nabla\times P_s^{k\omega}-ik\omega\mu Q_s^{k\omega}=0,\qquad k=1,2,\\
&\nabla\times Q_s^\omega+i\omega\varepsilon P_s^\omega+i\omega\chi^{(2)}\left(s\overline{E^\omega_1}\cdot F^{2\omega}_s+s\overline{F^\omega_s}\cdot E^{2\omega}_1+s^2\overline{F^\omega_s}\cdot F^{2\omega}_s\right)=0,\\
&\nabla\times Q_s^{2\omega}+i2\omega \varepsilon P_s^{2\omega}+i2\omega\chi^{(2)}\left(2s E^\omega_1\cdot F^\omega_s+s^2 F^\omega_s\cdot F^\omega_s\right)=0,\\
&\mathbf t(P_s^{k\omega})=0,\qquad k=1,2,
\end{split}\right.\]
which implies
\begin{equation}\label{eqn:est-P-Q}
\begin{split}
\sum_{k=1,2}\|P_s^{k\omega}&\|_{W^{1,p}(\Omega, \C^3)}+\|Q_s^{k\omega}\|_{W^{1,p}_{\Div}(\Omega)}\\
&\leq C|s|\left(\sum_{k=1,2}\|E^{k\omega}_1\|_{W^{1,p}_{\Div}(\Omega)}\right)\left(\sum_{k=1,2}\|F^{k\omega}_s\|_{W^{1,p}(\Omega, \C^3)}\right)+C|s|^2\sum_{k=1,2}\|F^{k\omega}_s\|^2_{W^{1,p}(\Omega;\C^3)}.
\end{split}
\end{equation}
Then by \eqref{eqn::estimate for E_1 and H_1} and \eqref{eqn::est-F-G}, we have
\[
\sum_{k=1,2}\|P_s^{k\omega}\|_{W^{1,p}(\Omega, \C^3)}+\|Q_s^{k\omega}\|_{W^{1,p}_{\Div}(\Omega)}
\leq C\sum_{k=1,2}\|f^{k\omega}\|_{TW^{1-1/p,p}_{\Div}(\p \Omega)}^2,
\]
which by \eqref{eqn:def-P-Q} and \eqref{eqn:est-E_2-H_2} provides
\begin{equation}\label{eqn:est-F-G-Cf}
\begin{split}
\sum_{k=1,2}\|F^{k\omega}_s\|_{W^{1,p}(\Omega;\C^3)}&+\|G^{k\omega}_s\|_{W^{1,p}_{\Div}(\Omega)}\\
&\leq C\sum_{k=1,2} \|f^{k\omega}\|_{TW^{1-1/p,p}_{\Div}(\p \Omega)}^2+\sum_{k=1,2}\|E^{k\omega}_2\|_{W^{1,p}(\Omega;\C^3)}+\|H^{k\omega}_2\|_{W^{1,p}_{\Div}(\Omega)}\\
&\leq C\sum_{k=1,2} \|f^{k\omega}\|_{TW^{1-1/p,p}_{\Div}(\p \Omega)}^2.
\end{split}
\end{equation}
We plug this back into \eqref{eqn:est-P-Q} to finally obtain
\begin{equation}\label{eqn:est-F-G-Cfs}
\begin{split}
\sum_{k=1,2}\|F_s^{k\omega}-E^{k\omega}_2&\|_{W^{1,p}(\Omega, \C^3)}+\|G_s^{k\omega}-H^{k\omega}_2\|_{W^{1,p}_{\Div}(\Omega)}\\
&\leq C|s|\sum_{k=1,2}\|f^{k\omega}\|_{TW^{1-1/p,p}_{\Div}(\p \Omega)}^3+C|s|^2\sum_{k=1,2}\|f^{k\omega}\|_{TW^{1-1/p,p}_{\Div}(\p \Omega)}^4\\
&\leq C_f|s|
\end{split}\end{equation}
for $|s|$ small enough and some constant $C_f>0$ depending on $(f^\omega, f^{2\omega})$. \\

{ Denote by $\Lambda^{\omega, 2\omega}_{\varepsilon,\mu}$ the admittance map $\Lambda^{\omega,2\omega}_{\varepsilon,\mu,0}$ for linear Maxwell's equations, i.e., 
\[\Lambda^{\omega, 2\omega}_{\varepsilon,\mu}(f^\omega, f^{2\omega})=(\mathbf t(H^\omega_1), \mathbf t(H^{2\omega}_1)).\]
We obtain the following asymptotic expansion of the admittance map.}
\begin{proposition}\label{prop::asymptotics of the admittance map}
Suppose that $(f^\omega, f^{2\omega})\in TW^{1-1/p,p}_{\Div}(\p \Omega)\times TW^{1-1/p,p}_{\Div}(\p \Omega)$ with $3<p\le 6$. Then
\begin{equation}\label{eqn::1st term in asymptotics of the admittance map}
s^{-1}[\Lambda^{\omega,2\omega}_{\varepsilon,\mu,\chi^{(2)}}(sf^\omega, sf^{2\omega})-s\Lambda^{\omega, 2\omega}_{\varepsilon,\mu}(f^\omega, f^{2\omega})]\to 0,
\end{equation}
and
\begin{equation}\label{eqn::2nd term in asymptotics of the admittance map}
s^{-2}[\Lambda^{\omega,2\omega}_{\varepsilon,\mu,\chi^{(2)}}(sf^\omega, sf^{2\omega})-s\Lambda^{\omega, 2\omega}_{\varepsilon,\mu}(f^\omega, f^{2\omega})]\to \left(\mathbf t(H^\omega_2), \mathbf t(H^{2\omega}_2)\right),
\end{equation}
both in $TW^{1-1/p,p}_{\Div}(\p \Omega)\times TW^{1-1/p,p}_{\Div}(\p \Omega)$ as $s\rightarrow0$.
\end{proposition}
\begin{proof}
From \eqref{eqn:def-F-G} we have
$$
\Lambda^{\omega, 2\omega}_{\varepsilon,\mu,\chi^{(2)}}(sf^\omega, sf^{2\omega})-s\Lambda^{\omega, 2\omega}_{\varepsilon,\mu}(f^\omega, f^{2\omega})=s^{2}\left(\mathbf t(G^{\omega}_s),\mathbf t(G^{2\omega}_s)\right).
$$
Then by boundedness of $\mathbf t$ from $W^{1,p}_{\Div}(\Omega)$ onto $TW^{1-1/p,p}_{\Div}(\p\Omega)$ and by \eqref{eqn:est-F-G-Cf},
\[\begin{split}
\|s^{-1}[\Lambda^{\omega, 2\omega}_{\varepsilon,\mu,\chi^{(2)}}(sf^\omega, sf^{2\omega})-s\Lambda^{\omega, 2\omega}_{\varepsilon,\mu}(f^\omega, f^{2\omega})\|_{TW^{1-1/p,p}_{\Div}(\p\Omega)\times TW^{1-1/p,p}_{\Div}(\p \Omega)}&\le C|s|\sum_{k=1,2}\|G^{k\omega}_s\|_{W^{1,p}_{\Div}(\Omega)}\\
&\le C_f|s|.
\end{split}\]
Taking $s\to 0$, this implies \eqref{eqn::1st term in asymptotics of the admittance map}. Similarly, by \eqref{eqn:est-F-G-Cfs},
\[\begin{split}
\|s^{-2}[\Lambda^{\omega,2\omega}_{\varepsilon,\mu,\chi^{(2)}}(sf^\omega, sf^{2\omega})-&s\Lambda^{\omega, 2\omega}_{\varepsilon,\mu}(f^\omega, f^{2\omega})]-\left(\mathbf t(H^\omega_2), \mathbf t(H^{2\omega}_2)\right)\|_{TW^{1-1/p,p}_{\Div}(\p\Omega)\times TW^{1-1/p,p}_{\Div}(\p\Omega)}\\
&\le C\sum_{k=1,2}\|G^{k\omega}_s-H^{k\omega}_2\|_{W^{1,p}_{\Div}(\Omega)}\\
&\le C_f{ |s|}, 
\end{split}\]
which proves \eqref{eqn::2nd term in asymptotics of the admittance map}.
\end{proof}

By \eqref{eqn::1st term in asymptotics of the admittance map} in the Proposition \ref{prop::asymptotics of the admittance map} and the result in \cite{OlaSomersalo1996}, the uniqueness of {$\mu, \varepsilon \in C^3(\Omega;\C)$} in Theorem \ref{thm::main result 2} is proved from the injectivity of the map 
\[(\mu, \varepsilon)~\mapsto~\Lambda^{\omega, 2\omega}_{\mu,\varepsilon}.\]  

\begin{remark}{\rm
One can see that the proof applies to the case when $\mu$ and $\varepsilon$ also depend on the frequency. That is, given $(\varepsilon^\omega, \mu^\omega, \varepsilon^{2\omega}, \mu^{2\omega})\in \left(C^2(\Omega;\C^3)\right)^4$ with positive real-parts, the admittance map 
\[\Lambda^{\omega, 2\omega}_{\varepsilon^\omega, \mu^\omega, \varepsilon^{2\omega}, \mu^{2\omega}, \chi^{(2)}}:~(f^\omega, f^{2\omega})~\mapsto~\left(\mathbf t(H^\omega), \mathbf t(H^{2\omega})\right)\] 
for the Maxwell's system
\[
\left\{
\begin{split}
&\nabla\times E^{\omega}-i\omega\mu^\omega H^{\omega}=0,\\
&\nabla\times H^{\omega}+i\omega\varepsilon^\omega E^{\omega}+i\omega\chi^{(2)}\overline{E^{\omega}}\cdot{E^{2\omega}}=0,\\
&\nabla\times E^{2\omega}-i2\omega\mu^{2\omega} H^{2\omega}=0,\\
&\nabla\times H^{2\omega}+i2\omega\varepsilon^{2\omega} E^{2\omega}+i2\omega\chi^{(2)}E^{\omega}\cdot E^{\omega}=0\\
&\mathbf t(E^{k\omega}) = f^{k\omega},\qquad k=1,2.
\end{split}
\right.
\]
uniquely determines $(\varepsilon^\omega, \mu^\omega, \varepsilon^{2\omega}, \mu^{2\omega})\in \left(C^2(\Omega;\C^3)\right)^4$ since the linear admittance map after first order approximation decouples to two linear admittance maps for linear Maxwell's equations associated to frequencies $\omega$ and $2\omega$. }
\end{remark}

\section{Proof of Theorem~\ref{thm::main result 2}}\label{sec:chi_2}

In this section we continue the proof of Theorem~\ref{thm::main result 2} by showing $\chi^{(2)}_1=\chi^{(2)}_2$. To that end, we shall use complex geometrical optics solutions in an integral identity~\eqref{eqn:key-identity}.

\subsection{An integral identity}
{ Now, by \eqref{eqn::2nd term in asymptotics of the admittance map} in Proposition~\ref{prop::asymptotics of the admittance map}, we obtain $\mathbf t(H^{k\omega}_{2,1})=\mathbf t(H^{k\omega}_{2,2})$ for $k=1,2$, where $(E^{\omega}_{2,j},H^\omega_{2,j}, E^{2\omega}_{2,j}, H^{2\omega}_{2,j})\in W^{1,p}_{D}(\Omega)\times W^{1,p}_{\Div}(\Omega)\times W^{1,p}_{D}(\Omega)\times W^{1,p}_{\Div}(\Omega)$, $j=1,2$, is the unique solution to
\begin{equation}\label{eqn::nonlinear Maxwell equation with j-indice}
\left\{\begin{split}
&\nabla\times E^{k\omega}_{2,j}-ik\omega\mu H^{k\omega}_{2,j}=0,\qquad k=1,2,\\
&\nabla\times H^\omega_{2,j}+i\omega\varepsilon E^\omega_{2,j}+i\omega \chi^{(2)}_j\overline{E^\omega_1}\cdot E^{2\omega}_1=0,\\
&\nabla\times H^{2\omega}_{2,j}+i2\omega \varepsilon E^{2\omega}_{2,j}+i2\omega \chi^{(2)}_j E^\omega_1\cdot E^\omega_1=0,\\
&\mathbf t(E^{k\omega}_{2,j})=0,\qquad k=1,2.
\end{split}
\right.\end{equation}
with $(E^\omega_1,H^\omega_1, E^{2\omega}_1, H^{2\omega}_1)\in W^{1,p}_{\Div}(\Omega)\times W^{1,p}_{\Div}(\Omega)\times W^{1,p}_{\Div}(\Omega)\times W^{1,p}_{\Div}(\Omega)$ is a solution to
\begin{equation}\label{eqn::linear Maxwell equation with epsilon and mu}
\nabla \times E^{k\omega}_1-ik\omega\mu H^{k\omega}_1=0,\quad
\nabla\times H^{k\omega}_1+ik\omega \varepsilon E^{k\omega}_1=0,\qquad k=1,2.
\end{equation}
satisfying $\mathbf t(E^{k\omega}_1)=f^{k\omega}$ with $k=1,2$. Let $(\wt E^\omega, \wt H^\omega, \wt E^{2\omega}, \wt H^{2\omega}) \in W^{1,p}_{\Div}(\Omega)\times W^{1,p}_{\Div}(\Omega)\times W^{1,p}_{\Div}(\Omega)\times W^{1,p}_{\Div}(\Omega)$ be a solution to
\begin{equation}\label{eqn::linear Maxwell equation with bar-epsilon and bar-mu}
\nabla\times \wt E^{k\omega}-ik\omega\overline \mu \wt H^{k\omega}=0,\quad
\nabla\times \wt H^{k\omega}+ik\omega\overline \varepsilon \wt E^{k\omega}=0,\qquad k=1,2.
\end{equation}
Using integration by parts,
\[\begin{split}
\sum_{k=1,2}\int_{\p\Omega}\mathbf t(H^{k\omega}_2)\cdot(\nu\times \mathbf t(\overline{\wt E^{k\omega}}))~dS =& \sum_{k=1,2}\int_\Omega \nabla\times H^{k\omega}_2\cdot \overline{\wt E^{k\omega}}-H^{k\omega}_2\cdot \nabla\times \overline{\wt E^{k\omega}}~dx\\
=& \int_\Omega \sum_{k=1,2} -ik\omega\varepsilon E^{k\omega}_2\cdot\overline{\wt E^{k\omega}}-i\omega\big(\chi^{(2)}\cdot \overline{\wt E^{\omega}}\big)\left(\overline{E_1^{\omega}}\cdot E^{2\omega}_1\right)\\
&-i2\omega \big(\chi^{(2)}\cdot \overline{\wt E^{2\omega}}\big)\big({E_1^{\omega}}\cdot E^{\omega}_1\big)-\sum_{k=1,2} H^{k\omega}_2\cdot \big(\overline{ik\omega\overline\mu \wt H^{k\omega}}\big)~dx,
\end{split}\]
then
\[\begin{split}
\sum_{k=1,2}\int_{\p\Omega}\mathbf t(H^{k\omega}_2)\cdot(\nu\times \mathbf t(\overline{\wt E^{k\omega}}))~dS =& \int_\Omega\sum_{k=1,2} E^{k\omega}_2\cdot\big(\overline{-\nabla\times \wt H^{k\omega}})-i\omega\big(\chi^{(2)}\cdot \overline{\wt E^{\omega}}\big)\left(\overline{E_1^{\omega}}\cdot E^{2\omega}_1\right)\\
&-i2\omega \big(\chi^{(2)}\cdot \overline{\wt E^{2\omega}}\big)\big({E_1^{\omega}}\cdot E^{\omega}_1\big)+\sum_{k=1,2} \big(\nabla\times E^{k\omega}_2\big)\cdot\overline{\wt H^{k\omega}}~dx\\
=& \int_{\p\Omega}\sum_{k=1,2}\mathbf t(E^{k\omega}_2)\cdot (\nu\times\overline{\mathbf t(\wt H^{k\omega})})~dS\\
&+\int_\Omega -i\omega\big(\chi^{(2)}\cdot \overline{\wt E^{\omega}}\big)\left(\overline{E_1^{\omega}}\cdot E^{2\omega}_1\right)-i2\omega \big(\chi^{(2)}\cdot \overline{\wt E^{2\omega}}\big)\big({E_1^{\omega}}\cdot E^{\omega}_1\big)~dx\\
=&\int_\Omega -i\omega\big(\chi^{(2)}\cdot \overline{\wt E^{\omega}}\big)\left(\overline{E_1^{\omega}}\cdot E^{2\omega}_1\right)-i2\omega \big(\chi^{(2)}\cdot \overline{\wt E^{2\omega}}\big)\big({E_1^{\omega}}\cdot E^{\omega}_1\big)~dx
\end{split}\]
For $\chi^{(2)}_j$, $j=1,2$,  we have 
\[\begin{split}
\sum_{k=1,2}\int_{\p\Omega}\mathbf t(H^{k\omega}_{2,j})\cdot(\nu\times \mathbf t(\overline{\wt E^{k\omega}}))~dS = -i\omega\int_\Omega \chi^{(2)}_j\cdot \Big[\left(\overline{E_1^{\omega}}\cdot E^{2\omega}_1\right)\overline{\wt E^{\omega}}+2\big({E_1^{\omega}}\cdot E^{\omega}_1\big)\overline{\wt E^{2\omega}}\Big]~dx.
\end{split}\]
Then by that $\mathbf t(H^{k\omega}_{2,1})=\mathbf t(H^{k\omega}_{2,2})$ for $k=1,2$, we have the key identity
\begin{equation}\label{eqn:key-identity}
\int_\Omega \left(\chi^{(2)}_1-\chi^{(2)}_2\right)\cdot \Big[\left(\overline{E_1^{\omega}}\cdot E^{2\omega}_1\right)\overline{\wt E^{\omega}}+2\big({E_1^{\omega}}\cdot E^{\omega}_1\big)\overline{\wt E^{2\omega}}\Big]~dx=0
\end{equation}
for solutions to \eqref{eqn::linear Maxwell equation with epsilon and mu} and \eqref{eqn::linear Maxwell equation with bar-epsilon and bar-mu}. 

\subsection{Construction of complex geometrical optics solutions}

In this section, we construct appropriated solutions to the linear equations \eqref{eqn::linear Maxwell equation with epsilon and mu} and \eqref{eqn::linear Maxwell equation with bar-epsilon and bar-mu} in order to plug in the integral identity \eqref{eqn:key-identity}. 
We adopt the approach in \cite{OPS1993}. For the completeness of the work, we present the main steps here. 

Given $\zeta\in\C^3$ such that $\zeta\cdot\zeta=\kappa^2$, the Faddeev kernel is defined as
\begin{equation}\label{eqn:Faddeev_kernel}g_\zeta(x)=(2\pi)^{-3}\int\frac{e^{i\xi\cdot x}}{|\xi|^2-2i\zeta\cdot\xi}~d\xi.\end{equation}
Then $\mathscr{G}(x)=e^{\zeta\cdot x}g_\zeta(x)$ is a fundamental solution to the operator $-(\Delta+\kappa^2)$ and 
\[\mathbf G=\frac{\kappa^2}{\omega}\left(\begin{array}{cc}1+\frac{\nabla\nabla\cdot}{\kappa^2} & \frac{i}{\omega\varepsilon_0}\nabla\times \\-\frac{i}{\omega\mu_0}\nabla\times & 1+\frac{\nabla\nabla\cdot}{\kappa^2}\end{array}\right)\mathscr{G}\]
is a fundamental solution to the Maxwell's operator $\mathscr{L}-\omega$ in vacuum, where 
\[\mathscr{L}=\left(\begin{array}{cc}0 & \frac{i}{\varepsilon_0}\nabla\times \\-\frac{i}{\mu_0}\nabla\times & 0\end{array}\right),\qquad \kappa^2=\omega^2\mu_0\varepsilon_0.\]
Here the operators $\nabla\nabla\cdot$ and $\nabla\times$ are interpreted as matrix operators acting on scalar functions
\[\nabla\nabla\cdot=\left(\begin{array}{ccc}\partial_{11} & \partial_{12} & \partial_{13} \\\partial_{21} & \partial_{22} & \partial_{23} \\\partial_{31} & \partial_{32} & \partial_{33}\end{array}\right),\qquad 
\nabla\times =\left(\begin{array}{ccc}0 & -\partial_3 & \partial_2 \\\partial_3 & 0 & -\partial_1 \\-\partial_2 & \partial_1 & 0\end{array}\right).\]
As a result, Maxwell's equations for $(E, H)$ in the whole space
\[\nabla\times E-i\omega\mu H=0,\quad \nabla\times H+i\omega\varepsilon E=0\]
read
\begin{equation}\label{eqn:ME_L}(\mathscr{L}-\omega)\left(\begin{array}{c}E \\H\end{array}\right)=\omega\left(\begin{array}{c}\wt\varepsilon E \\\wt\mu H\end{array}\right)\end{equation}
where $\wt\varepsilon=(\varepsilon-\varepsilon_0)/\varepsilon_0$ and $\wt\mu=(\mu-\mu_0)/\mu_0$. Given above fundamental solutions, the solutions we are looking for in turn solve the integral equation
\begin{equation}\label{eqn:ME_int}\left(\begin{array}{c}E \\H\end{array}\right)=\left(\begin{array}{c}E_0 \\H_0\end{array}\right)+\omega\int \mathbf G(\cdot-y)\left(\begin{array}{c}\wt\varepsilon(y) E(y) \\\wt\mu(y) H(y)\end{array}\right)~dy\end{equation}
for some vacuum solutions $(E_0, H_0)$ to 
\[(\mathscr L-\omega)\left(\begin{array}{c}E_0 \\H_0\end{array}\right)=0.\]
Specifically, we choose 
\[E_0(x)=e^{\zeta\cdot x}A_\zeta,\qquad H_0(x)=e^{\zeta\cdot x}B_\zeta\]
where $\zeta\times A_\zeta=\omega \mu_0 B_\zeta$ and $\zeta\times B_\zeta=-\omega\varepsilon_0 A_\zeta$. In particular, if $\zeta\cdot A_\zeta=0$, one only needs to choose $A_\zeta=\frac1{-\omega\varepsilon_0}\zeta\times B_\zeta$. \\

By \cite[Theorem 2.5]{OPS1993}, for {$\mu, \varepsilon\in C^3(\Omega, \C)$}, the equation \eqref{eqn:ME_L}, or equivalently \eqref{eqn:ME_int}, admits a unique solution $(E, H)$ of the form 
\[E(x)=e^{\zeta\cdot x}(A_\zeta+R(x)),\qquad H(x)=e^{\zeta\cdot x}(B_\zeta+Q(x)),\]
with $R(x)$ and $Q(x)$ in $L^2_{-\delta}(\R^3; \C^3)$ for $\delta\in [1/2, 1]$, where
\[L^2_\delta(\R^3;\C^n):=\left\{ f\in L^2_{\textrm{loc}}(\R^3;\C^n)~\Big|~\|f\|_\delta=\left(\int_{\R^3}(1+|x|^2)^\delta|f(x)|^2~dx\right)^{1/2}<\infty\right\}.\]
Then by \cite[Lemma 2.4]{OPS1993}, $(R(x), Q(x))$ satisfies the equation
\begin{equation}\label{eqn:RQ_int}
\left(\begin{array}{c}R \\Q\end{array}\right)=\left(\begin{array}{c}J \\K\end{array}\right)+M^{-1}G_\zeta\left(M(V-q)\left(\begin{array}{c}R \\Q\end{array}\right)\right)
\end{equation}
where 
\begin{equation}\label{eqn:JK_int}
\left(\begin{array}{c}J \\K\end{array}\right)=(M^{-1}M_0-1)\left(\begin{array}{c}A_\zeta \\B_\zeta\end{array}\right)+M^{-1}G_\zeta\left(M(V-q)\left(\begin{array}{c}A_\zeta \\B_\zeta\end{array}\right)\right),\end{equation}
\[M=\left(\begin{array}{cc}\varepsilon^{1/2} & 0 \\0 & \mu^{1/2}\end{array}\right),\qquad M_0=\left(\begin{array}{cc}\varepsilon_0^{1/2} & 0 \\0 & \mu_0^{1/2}\end{array}\right),\]
and
\[V=\left(\begin{array}{cc}\omega^2(\mu\varepsilon-\mu_0\varepsilon_0)+\nabla^2\log\varepsilon & \frac{i\omega}{\varepsilon}\nabla(\mu\varepsilon)\times \\-\frac{i\omega}{\mu}\nabla(\mu\varepsilon)\times & \omega^2(\mu\varepsilon-\mu_0\varepsilon_0)+\nabla^2\log\mu\end{array}\right),\quad q=\left(\begin{array}{cc}\frac{\Delta\varepsilon^{1/2}}{\varepsilon^{1/2}} & 0 \\0 & \frac{\Delta\mu^{1/2}}{\mu^{1/2}}\end{array}\right).\]
The operator $G_\zeta$ is the convolution of components with the Faddeev kernel $g_\zeta$ defined in \eqref{eqn:Faddeev_kernel}. It is a bounded operator from $L^2_\delta(\R^3;\C^6)$ to $L^2_{-\delta}(\R^3;\C^6)$ for $\delta\in (1/2, 1)$ with the decaying property (see \cite[Proposition 2.1]{OPS1993})
\begin{equation}\label{eqn:Faddeev_decay}
\|G_\zeta f\|_{-\delta}\leq \frac C{|\zeta|}\|f\|_\delta, \qquad f\in L^2_\delta(\R^3;\C^6).\end{equation}
The proof of such decay can be found in \cite{SylvesterUhlmann1987}. 

Denote $E=e^{\zeta\cdot x}\mathbf e$ and $H=e^{\zeta\cdot x}\mathbf h$, that is, $\mathbf e(x)=A_\zeta+R(x)$ and $\mathbf h(x)=B_\zeta+Q(x)$. By \eqref{eqn:RQ_int} and \eqref{eqn:JK_int}, we have
\begin{equation}\label{eqn:eh_int}
M\left(\begin{array}{c}\mathbf e \\ \mathbf h \end{array}\right)=M_0\left(\begin{array}{c}A_\zeta \\B_\zeta\end{array}\right)+G_\zeta\left(M(V-q)\left(\begin{array}{c}\mathbf e \\\mathbf h\end{array}\right)\right).\end{equation}
Using the decaying estimate \eqref{eqn:Faddeev_decay} for $G_\zeta$, for $|\zeta|$ large, \eqref{eqn:eh_int} admits a solution by Neumann series
\begin{equation}\label{eqn:eh}\left(\begin{array}{c}\mathbf e \\\mathbf h\end{array}\right)=M^{-1}M_0\left(\left(\begin{array}{c}A_\zeta \\B_\zeta\end{array}\right)+\left(\begin{array}{c}\wt R \\\wt Q\end{array}\right)\right)=\left(\begin{array}{c}\varepsilon_0^{1/2}\varepsilon^{-1/2}(A_\zeta+\wt R) \\\mu_0^{1/2}\mu^{-1/2}(B_\zeta+\wt Q)\end{array}\right)\end{equation}
where $\wt R, \wt Q\in L^2_{-\delta}(\R^3;\C^3)$ satisfying
\[\|\wt R\|_{-\delta}, \|\wt Q\|_{-\delta}\leq \frac {C (|A_\zeta|+|B_\zeta|)}{|\zeta|}, \qquad \delta\in (1/2, 1)\]
where $C>0$ is independent of $|\zeta|$. Note that $R$ and $Q$ do not satisfy this decaying property since
\[\left(\begin{array}{c}R \\Q\end{array}\right)=\left(\begin{array}{c}\left(\varepsilon_0^{1/2}\varepsilon^{-1/2}-1\right)A_\zeta+\varepsilon_0^{1/2}\varepsilon^{-1/2}\wt R \\\left(\mu_0^{1/2}\mu^{-1/2}-1\right)B_\zeta+\mu_0^{1/2}\mu^{-1/2}\wt Q\end{array}\right).\]

Moreover, by \cite[Lemma 2.11 (b)]{Nachman1988},
 we have that the solution $(E, H)=(e^{\zeta\cdot x}\mathbf e, e^{\zeta\cdot x}\mathbf h)$, with $(\mathbf e, \mathbf h)$ given by \eqref{eqn:eh}, belongs to $H^2(\Omega; \C^3)\times H^2(\Omega; \C^3)$, and 
\begin{equation}\label{eqn:wtRQ_Hs}\|\wt R\|_{H^s(\Omega;\C^3)}, \|\wt Q\|_{H^s(\Omega;\C^3)}\leq \frac {C (|A_\zeta|+|B_\zeta|)}{|\zeta|^{1-s}},\qquad s\in[0,1].\end{equation}
By Sobolev embedding, we have $(E, H)\in W^{1,p}(\Omega;\C^3)\times W^{1,p}(\Omega;\C^3)$ for $2\leq p\leq 6$. Then by Maxwell's equations and $\Div(\mathbf t(E))=\nu\cdot (\nabla\times E)$, we obtain that $(E, H)\in W_{\Div}^{1,p}(\Omega)\times W_{\Div}^{1,p}(\Omega)$. Finally, one can obtain the estimates
\begin{equation}\label{eqn:wtRQ_Lp}\|\wt R\|_{L^p(\Omega;\C^3)}, \|\wt Q\|_{L^p(\Omega;\C^3)}\leq \frac {C(|A_\zeta|+|B_\zeta|)}{|\zeta|^{\frac{6-p}{2p}}}\end{equation}
using the inequality $0\leq \frac{3p-6}{2p}\leq 1$, Sobolev embedding and \eqref{eqn:wtRQ_Hs}. 
To summarize, we obtain 
\begin{proposition}\label{prop:CGO}
Let $\Omega$ be a bounded domain in $\R^3$ with smooth boundary and $2\leq p\leq 6$. Suppose that {$\varepsilon, \mu\in C^3(\Omega;\C)$} with positive real parts in $\Omega$. Then for $\zeta\in \C^3$ with $\zeta\cdot\zeta=\kappa^2=\omega^2\mu_0\varepsilon_0$, and $A_\zeta, B_\zeta$ satisfying $\zeta\times A_\zeta=\omega \mu_0 B_\zeta$ and  $\zeta\times B_\zeta=-\omega\varepsilon_0 A_\zeta$, with $|\zeta|$ large enough, the Maxwell's equations
\[\nabla\times E=i\omega\mu H, \qquad \nabla\times H=-i\omega\varepsilon E\]
has a solution $(E, H)\in W^{1,p}_{\Div}(\Omega;\C^3)\times W^{1,p}_{\Div}(\Omega;\C^3)$ of the form
\[E(x)=e^{\zeta\cdot x}\varepsilon_0^{1/2}\varepsilon^{-1/2}(A_\zeta+\wt R),\quad H(x)=e^{\zeta\cdot x}\mu_0^{1/2}\mu^{-1/2}(B_\zeta+\wt Q)\]
where $\wt R, \wt Q$ satisfy \eqref{eqn:wtRQ_Lp}.
\end{proposition}
{\begin{remark}
In the following we choose $\zeta$ and $A_\zeta$ satisfying $\zeta\cdot\zeta=\kappa^2$ and $\zeta\cdot A_\zeta=0$. Then $\zeta\times (\zeta\times A_\zeta)=-\kappa^2 A_\zeta$, which is equivalent to $\zeta\times A_\zeta=\omega \mu_0 B_\zeta$ and  $\zeta\times B_\zeta=-\omega\varepsilon_0 A_\zeta$.
\end{remark}}

\subsection{Proof of Theorem~\ref{thm::main result 2}: Part II }

Recall that we assume $3<p<6$. For arbitrary $\xi\in \R^3\backslash\{0\}$, choose coordinates such that $\xi=\xi_1\bf e_1$. Set
\[\begin{split}\zeta_1^\omega&=i\frac{\xi_1}{2}{\bf e_1}-\sqrt{\frac{\xi_1^2}{4}+\tau^2}~{\bf e_2}+i\sqrt{\tau^2-\kappa^2}~{\bf e_3}\sim \tau(-{\bf e_2}+i{\bf e_3}).\\
\zeta_1^{2\omega}&=-i\xi_1{\bf e_1}+2\sqrt{\frac{\xi_1^2}{4}+\tau^2}~{\bf e_2}+2i\sqrt{\tau^2-\kappa^2}~{\bf e_3}\sim \tau(2{\bf e_2}+2i{\bf e_3}),\\
\wt \zeta^\omega&=-i\frac{\xi_1}{2}{\bf e_1}-\sqrt{\frac{\xi_1^2}{4}+\tau^2}~{\bf e_2}+i\sqrt{\tau^2-\kappa^2}~{\bf e_3}\sim \tau(-{\bf e_2}+i{\bf e_3}),\\
\wt\zeta^{2\omega}&=i\xi_1{\bf e_1}+2\sqrt{\frac{\xi_1^2}{4}+\tau^2}~{\bf e_2}+2i\sqrt{\tau^2-\kappa^2}~{\bf e_3}\sim \tau(2{\bf e_2}+2i{\bf e_3}).
\end{split}\]
Then we have
\[\zeta_1^\omega\cdot\zeta_1^\omega=\wt\zeta^\omega\cdot\wt\zeta^\omega=\kappa^2,\quad \zeta_1^{2\omega}\cdot\zeta_1^{2\omega}=\wt\zeta^{2\omega}\cdot \wt\zeta^{2\omega} =4\kappa^2=(2\omega)^2\mu_0\varepsilon_0\]
and 
$$
\overline{\zeta_1^\omega}+\zeta_1^{2\omega}+\overline{\wt\zeta^\omega}=-i\xi,\quad 2\zeta_1^\omega + \overline{\wt\zeta^{2\omega}} = 0.
$$

{
Then we pick
\[
A_1^\omega={\bf e_1}+ \frac{-i\frac{\xi_1}{2}-i\sqrt{\tau^2-\kappa^2}}{-\sqrt{\frac{\xi_1^2}{4}+\tau^2}}{\bf e_2}+{\bf e_3},\qquad
A_1^{2\omega}={\bf e_1}+{\bf e_2}+\frac{i\frac{\xi_1}{2}-\sqrt{\frac{\xi_1^2}{4}+\tau^2}}{i\sqrt{\tau^2-\kappa^2}}{\bf e_3},
\]
and $\wt A^{2\omega}=0$. 
One can verify 
\[\zeta_1^\omega\cdot A_1^\omega=\zeta_1^{2\omega}\cdot A_1^{2\omega}=\wt\zeta^{2\omega}\cdot\wt A^{2\omega}=0.\]

For $\wt A^\omega$, we have two choices
\[\wt A^{\omega}={\bf e_2}+ \frac{\sqrt{\frac{\xi_1^2}{4}+\tau^2}}{i\sqrt{\tau^2-\kappa^2}}{\bf e_3}\quad\textrm{ or }\quad\wt A^{\omega}={\bf e_1}+{\bf e_2}+\frac{\frac{i\xi_1}{2}+\sqrt{\frac{\xi_1^2}{4}+\tau^2}}{i\sqrt{\tau^2-\kappa^2}}{\bf e_3}.\]
Then $\wt A^\omega\cdot\wt\zeta^\omega=0$.

Then by Proposition \ref{prop:CGO}, we have solutions to the linear equations \eqref{eqn::nonlinear Maxwell equation with j-indice} and \eqref{eqn::linear Maxwell equation with bar-epsilon and bar-mu} given by 
\[\begin{split}E^\omega_1&=e^{\zeta_1^\omega\cdot x}\varepsilon_0^{1/2}\varepsilon^{-1/2}(A_1^\omega+R_1^\omega), \\
E^{2\omega}_1&=e^{\zeta_1^{2\omega}\cdot x}\varepsilon_0^{1/2}\varepsilon^{-1/2}(A_1^{2\omega}+R_1^{2\omega}),\\
\wt E^\omega&=e^{\wt \zeta^\omega\cdot x}\varepsilon_0^{1/2}\varepsilon^{-1/2}(\wt A^\omega+\wt R^\omega),\\
\wt E^{2\omega}&=e^{\wt\zeta^{2\omega}\cdot x}\varepsilon_0^{1/2}\varepsilon^{-1/2}(\wt A^{2\omega}+\wt R^{2\omega}),
\end{split}\]
where 
\[\|R_1^\omega\|_{L^p(\Omega;\C^3)}, \|R_1^{2\omega}\|_{L^p(\Omega;\C^3)}, \|\wt R^\omega\|_{L^p(\Omega;\C^3)}, \|\wt R^{2\omega}\|_{L^p(\Omega;\C^3)} \leq \frac C{\tau^{\frac{6-p}{2p}}}.\]
Since we assume $p<6$, plugging these solutions into the integral identity \eqref{eqn:key-identity} and using the generalized H\"older's inequality and above decay property for the terms involving remainders, we obtain as $\tau\rightarrow\infty$
\[\int e^{-i\xi\cdot x} a \chi_\Omega\left(\chi^{(2)}_1-\chi^{(2)}_2\right)\cdot ({\bf e_2}+i{\bf e_3})~dx=0\]
and
\[\int e^{-i\xi\cdot x} a \chi_\Omega\left(\chi^{(2)}_1-\chi^{(2)}_2\right)\cdot ({\bf e_1}+{\bf e_2}+i{\bf e_3})~dx=0\]
where ${\bf e_1}=\hat\xi=\xi/|\xi|$, and ${\bf e_2}$ and ${\bf e_3}$ are two orthogonal directions that are perpendicular to ${\bf e_1}$.
Here and in what follows, $a:=\varepsilon_0^{3/2}|\varepsilon|^{-1}\overline{\varepsilon}^{-1/2}$. 
Hence, we obtain 
\begin{equation}\label{eqn:FT_1}\mathcal F\left(a \chi_\Omega(\chi^{(2)}_1-\chi^{(2)}_2)\right)(\xi)\cdot ({\bf e_2}+i{\bf e_3})=0\quad\textrm{ and }\quad \mathcal F\left(a \chi_\Omega(\chi^{(2)}_1-\chi^{(2)}_2)\right)(\xi)\cdot {\bf e_1}=0\end{equation}
for all $\xi\in\R^3\backslash\{0\}$, where $\mathcal F$ is the Fourier transform. 

Then we switch the ${\bf e_2}$ and $\bf e_3$ components in $\zeta_1^\omega$, $\zeta_1^{2\omega}$, $\wt\zeta^\omega$, $\wt\zeta^{2\omega}$, $A_1^\omega$, $A^{2\omega}_1$
and therein choose 
\[\wt A^{\omega}= \frac{\sqrt{\frac{\xi_1^2}{4}+\tau^2}}{i\sqrt{\tau^2-\kappa^2}}{\bf e_2}+{\bf e_3}.\]
Similar to above, we derive
\[\mathcal F\left(a \chi_\Omega(\chi^{(2)}_1-\chi^{(2)}_2)\right)(\xi)\cdot (i{\bf e_2}+{\bf e_3})=0,\]
which together with the first equation in \eqref{eqn:FT_1} implies that for complex-valued $a \chi_\Omega(\chi^{(2)}_1-\chi^{(2)}_2)$, we have 
\[\mathcal F\left(a \chi_\Omega(\chi^{(2)}_1-\chi^{(2)}_2)\right)(\xi)\cdot \mu=0\]
for any $\mu\in\R^3$ satisfying $\mu\perp\xi$. Writing $\xi=(\xi_1,\xi_2,\xi_3)$, let $\mu_{jk}=\xi_je_k-\xi_ke_j$ (where $\{e_1, e_2, e_3\}$ is the standard basis of $\R^3$) for $j\neq k$. Then, we have for $j\neq k$
\[\xi_j\left[\mathcal F\left(a \chi_\Omega(\chi^{(2)}_1-\chi^{(2)}_2)\right)(\xi)\right]_k-\xi_k\left[\mathcal F\left(a \chi_\Omega(\chi^{(2)}_1-\chi^{(2)}_2)\right)(\xi)\right]_j=0. \] 
Therefore, we obtain in the sense of distributions
$$\curl \left(a \chi_\Omega(\chi^{(2)}_1-\chi^{(2)}_2)\right)=0$$ in $\R^3$.

Similarly by the second equation in \eqref{eqn:FT_1}, we have $\mathcal F\left(a \chi_\Omega(\chi^{(2)}_1-\chi^{(2)}_2)\right)(\xi)\cdot\xi=0$, hence
\[\nabla\cdot \left(a \chi_\Omega(\chi^{(2)}_1-\chi^{(2)}_2)\right)=0.\]
Then we can conclude $a \chi_\Omega(\chi^{(2)}_1-\chi^{(2)}_2)=0$ in $\R^3$. Since $a$ is assumed non-vanishing in $\Omega$, this implies $\chi^{(2)}_1=\chi^{(2)}_2$ in $\Omega$. This completes the proof of Theorem~\ref{thm::main result 2}. 
}

\end{document}